\documentclass[12pt]{article}

\usepackage[margin=1in]{geometry}
\usepackage{CJKutf8, amsmath, amssymb, amsthm, authblk, xcolor, hyperref, physics}
\usepackage{cleveref}

\newtheorem{corollary}{Corollary}
\newtheorem{lemma}{Lemma}
\newtheorem{theorem}{Theorem}
\newtheorem{fact}{Fact}

\theoremstyle{definition}
\newtheorem{definition}{Definition}

\DeclareMathOperator*{\e}{\mathbb E}
\DeclareMathOperator{\Var}{Var}

\usepackage[style=trad-unsrt, citestyle=numeric-comp, giveninits=true, doi=false, url=false, eprint=false, isbn=false]{biblatex}
\begin{document}

\title{Improved concentration of Laguerre and Jacobi ensembles}

\author[1,2]{Yichen Huang (黄溢辰)\thanks{yichenhuang@fas.harvard.edu}}
\author[1]{Aram W. Harrow\thanks{aram@mit.edu}}
\affil[1]{Center for Theoretical Physics, Massachusetts Institute of Technology, Cambridge, Massachusetts 02139, USA}
\affil[2]{Department of Physics, Harvard University, Cambridge, Massachusetts 02138, USA}

\begin{CJK}{UTF8}{gbsn}

\maketitle

\end{CJK}

\begin{abstract}

We consider the asymptotic limits where certain parameters in the definitions of the Laguerre and Jacobi ensembles diverge. In these limits, Dette, Imhof, and Nagel proved that up to a linear transformation, the joint probability distributions of the ensembles become more and more concentrated around the zeros of the Laguerre and Jacobi polynomials, respectively. In this paper, we improve the concentration bounds. Our proofs are similar to those in the original references, but the error analysis is improved and arguably simpler. For the first and second moments of the Jacobi ensemble, we further improve the concentration bounds implied by our aforementioned results.

\end{abstract}

Preprint number: MIT-CTP/5469

\section{Introduction}

The Gaussian, Wishart, and Jacobi ensembles are three classical ensembles in random matrix theory. They find numerous applications in physics, statistics, and other branches of applied science. The Gaussian (Wishart) ensemble is also known as the Hermite (Laguerre) ensemble due to its relationship with the Hermite (Laguerre) polynomial.

Of particular interest are the asymptotic limits where certain parameters in the definitions of the ensembles diverge. In these limits, Dette, Imhof, and Nagel \cite{DI07, DN09} proved that up to a linear transformation, the joint probability distributions of the Hermite, Laguerre, and Jacobi ensembles become more and more concentrated around the zeros of the Hermite, Laguerre, and Jacobi polynomials, respectively. These results allow us to transfer knowledge on the zeros of orthogonal polynomials to the corresponding ensembles.

In this paper, we improve the concentration bounds for the Laguerre and Jacobi probability distributions around the zeros of the Laguerre and Jacobi polynomials, respectively. Our proofs are similar to those in the original references \cite{DI07, DN09}, but the error analysis is improved and arguably simpler. We also prove the concentration of the first and second moments of the Jacobi ensemble. The last result has found applications in quantum statistical mechanics \cite{HH22TET}.

The rest of this paper is organized as follows. Section \ref{s:r} presents our main results, which are compared with previous results in the literature. Proofs are given in Section \ref{s:p}.

\section{Results} \label{s:r}

In the literature, there is more than one definition of the Laguerre probability distribution. These definitions differ only by a linear transformation and are thus essentially equivalent. In this paper, we stick to one definition. When citing a result from the literature, we perform a linear transformation such that the result is presented for the definition we stick to. The same applies to the Jacobi case.

Let $n$ be the number of random variables in an ensemble. Let $\beta$ be the Dyson index, which can be an arbitrary positive number.

\subsection{Laguerre ensemble}\label{sec:laguerre}

We draw $\lambda_1\le\lambda_2\le\cdots\le\lambda_n$ from the Laguerre ensemble.

\begin{definition} [Laguerre ensemble] \label{def:le}
The probability density function of the $\beta$-Laguerre ensemble with parameters
\begin{equation} \label{eq:constr}
\alpha>(n-1)\frac{\beta}{2}
\end{equation}
is
\begin{equation} \label{eq:defl}
f_\textnormal{Lag}(\lambda_1,\lambda_2,\ldots,\lambda_n)\propto{\prod_{1\le i<j\le n}|\lambda_i-\lambda_j|^\beta}\prod_{i=1}^n\lambda_i^{\alpha-\frac{(n-1)\beta}{2}-1}e^{-\lambda_i/2},\quad\lambda_i>0.
\end{equation}
\end{definition}

For certain values of $\beta$, the Laguerre ensemble arises as the probability density function of the eigenvalues of a Wishart matrix $VV^*$, where $V$ is an $n\times \frac{2\alpha}{\beta}$ matrix with real ($\beta=1$), complex ($\beta=2$), or quaternionic ($\beta=4$) entries.  In each case the entries of $V$ are independent standard Gaussian random variables and $V^*$ denotes the conjugate transpose of $V$.

Let
\begin{equation}
    L_n^{(p)}(x):=\sum_{i=0}^n{n+p\choose n-i}\frac{(-x)^i}{i!},\quad p>-1
\end{equation}
be the Laguerre polynomial, whose zeros are all in the interval with endpoints \cite{IL92}
\begin{equation}
2n+p-2\pm\sqrt{1+4(n-1)(n+p-1)\cos^2\frac\pi{n+1}}.
\end{equation}
Let $x_1<x_2<\cdots<x_n$ be the zeros of the Laguerre polynomial $L_n^{(2\alpha/\beta-n)}(x/\beta)$.

We are interested in the limit $\alpha\to\infty$ but do not assume that $n\to\infty$. Note that if $\beta$ is a constant, then $n\to\infty$ implies that $\alpha\to\infty$; see (\ref{eq:constr}).

\begin{theorem} [Theorem 2.1 in Ref.~\cite{DI07}] \label{thm:l1}
For any $0<\epsilon<1$,
\begin{equation} \label{eq:thm1}
\Pr\left(\frac1{2\alpha}\max_{1\le i\le n}|\lambda_i-x_i|>\epsilon\right)\le4n(1+\epsilon^2/25)^\alpha e^{-\alpha\epsilon^2/25}.
\end{equation}
\end{theorem}

This theorem can be restated as

\begin{corollary} \label{c:1}
There exist positive constants $C_1,C_2$ such that for any $0<\epsilon<1$,
\begin{equation} \label{eq:cor1}
\Pr\left(\frac1{2\alpha}\max_{1\le i\le n}|\lambda_i-x_i|>\epsilon\right)\le C_1ne^{-C_2\alpha\epsilon^4}.
\end{equation}
\end{corollary}

\begin{theorem} [Theorem 2.4 in Ref.~\cite{DI07}] \label{thm:l2}
Let $\kappa\ge1$ be a parameter. If
\begin{equation} \label{eq:cond}
n-1+1/\beta\le2\alpha/\beta\le n-1+\kappa\quad\textnormal{and}\quad2\kappa\beta/\alpha<\epsilon<1,
\end{equation}
then there exist positive constants $C_1,C_2,C_3$ such that
\begin{equation} \label{eq:thm2}
\Pr\left(\frac1{2\alpha}\max_{1\le i\le n}|\lambda_i-x_i|>\epsilon\right)\le C_1n(e^{-C_2\alpha\epsilon^2/\kappa}+e^{C_3\kappa^2\beta-C_2\alpha\epsilon^2}).
\end{equation}
\end{theorem}

The original upper bound on $\Pr(\frac1{2\alpha}\max_{1\le i\le n}|\lambda_i-x_i|>\epsilon)$ in Theorem 2.4 of Ref.~\cite{DI07} is a complicated expression without implicit constants. The right-hand side of (\ref{eq:thm2}) is its simplification using implicit constants.

If condition (\ref{eq:cond}) is satisfied, (\ref{eq:thm2}) may be an improvement of (\ref{eq:cor1}). In particular, for a constant $\beta$, the right-hand side of (\ref{eq:thm2}) becomes $C'_1ne^{-C'_2\alpha\epsilon^2}$ ($C'_1,C'_2$ are positive constants) if and only if $\kappa$ is upper bounded by a constant.

As the main result of this subsection, Theorem \ref{thm:l3} is an improvement of Corollary \ref{c:1} and Theorem \ref{thm:l2}.

\begin{theorem} \label{thm:l3}
There exist positive constants $C_1,C_2$ such that for any $\epsilon>0$,
\begin{equation}
\Pr\left(\frac1{2\alpha}\max_{1\le i\le n}|\lambda_i-x_i|>\epsilon\right)\le C_1ne^{-C_2\alpha\epsilon\min\{\epsilon,1\}}.
\end{equation}
\end{theorem}

Let $n\le s$ be two positive integers and $V$ be an $n\times s$ matrix whose elements are independent standard real Gaussian random variables. Then, $VV^T$ is a real Wishart matrix, whose joint eigenvalue distribution is given by (\ref{eq:defl}) with $\beta=1$ and $\alpha=s/2$. Theorem \ref{thm:l3} implies that

\begin{corollary} \label{cor:l}
Let $\lambda_1\le\lambda_2\le\cdots\le\lambda_n$ be the eigenvalues of $VV^T$ and $x_1<x_2<\cdots<x_n$ be the zeros of the Laguerre polynomial $L_n^{(s-n)}(x)$. There exist positive constants $C_1,C_2$ such that for any $\epsilon>0$,
\begin{equation}
\Pr\left(\frac1s\max_{1\le i\le n}|\lambda_i-x_i|>\epsilon\right)\le C_1ne^{-C_2s\epsilon\min\{\epsilon,1\}}.
\end{equation}
\end{corollary}

Analogues of Corollary \ref{cor:l} for complex ($\beta=2$) and quaternionic ($\beta=4$) Wishart matrices also follow directly from Theorem \ref{thm:l3}.

Let
\begin{equation}
M_1^\textnormal{L} := \frac{1}{n} \sum_{i=1}^n \lambda_i
\end{equation}
be the first moment of the Laguerre ensemble. The distribution of $M_1^\textnormal{L}$ has a particularly simple form.

\begin{fact} \label{fact:M1L}
$M_1^\textnormal{L}$ is distributed as $\frac 1n \chi_{2\alpha n}^2$, where $\chi_k^2$ denotes the chi-square distribution with $k$ degrees of freedom.
\end{fact}

Thus, the concentration of $M_1^\textnormal{L}$ follows directly from the tail bound \cite{LM00, IL06} for the chi-square distribution.

The distribution of the second moment of the Laguerre ensemble does not have a simple form. Furthermore, it is complicated to obtain concentration bounds for the distribution, so we omit this analysis here.

\subsection{Jacobi ensemble}

We draw $\mu_1\le\mu_2\le\cdots\le\mu_n$ from the Jacobi ensemble.

\begin{definition} [Jacobi ensemble] \label{def:je}
The probability density function of the $\beta$-Jacobi ensemble with parameters $a,b>0$ is
\begin{equation}
f_\textnormal{Jac}(\mu_1,\mu_2,\ldots,\mu_n)\propto{\prod_{1\le i<j\le n}|\mu_i-\mu_j|^\beta}\prod_{i=1}^n(1-\mu_i)^{a-1}(1+\mu_i)^{b-1},\quad-1\le\mu_i\le1.
\end{equation}
\end{definition}

The Jacobi ensemble can be interpreted as the probability density function of the eigenvalues of a random matrix ensemble.  In the complex ($\beta=2$) case, let $Q_1$ and $Q_2$ be uniformly random projectors in $\mathbb C^{2n+a+b-2}$ with ranks $n$ and $n+b-1$, respectively. Then, $\frac{1+\mu_1}{2},\frac{1+\mu_2}{2}, \ldots, \frac{1+\mu_n}{2}$ are the non-zero eigenvalues of $Q_1Q_2Q_1$ \cite{Col05}. Equivalently, they are the squared singular values of an $n \times (n+b-1)$ rectangular block within a Haar-random unitary matrix of dimension $2n+a+b-2$.  A random matrix interpretation for general $\beta$ is given in Ref.~\cite{KN04}, but it has less of a natural connection to applications.

The Jacobi polynomial is defined as
\begin{equation} \label{eq:jacpoly}
    P_n^{p,q}(y):=\frac{\Gamma(n+p+1)}{\Gamma(n+p+q+1)}\sum_{i=0}^n\frac{\Gamma(n+p+q+i+1)}{i!(n-i)!\Gamma(p+i+1)}\left(\frac{y-1}2\right)^i,
  \end{equation}
where $\Gamma$ is the gamma function. It is well known that all zeros of the Jacobi polynomial are in the interval $(-1,1)$. Let $y_1<y_2<\cdots<y_n$ be the zeros of the Jacobi polynomial $P_n^{2a/\beta-1,2b/\beta-1}(y)$.

\subsubsection{Pointwise approximation}

In this subsubsection, we are interested in the limit $a+b\to\infty$ but do not assume that $\min\{a,b\}\to\infty$.

\begin{theorem} [Theorem 2.1 in Ref.~\cite{DN09}] \label{thm:j1}
For any $0<\epsilon\le1/2$,
\begin{equation}
\Pr\left(\max_{1\le i\le n}|\mu_i-y_i|>\epsilon\right)\le4(2n-1)\left(1+\frac{\epsilon^2}{162+2\epsilon^2}\right)^{a+b}e^{-\frac{(a+b)\epsilon^2}{162+2\epsilon^2}}.
\end{equation}
\end{theorem}

This theorem can be restated as

\begin{corollary} \label{c:3}
There exist positive constants $C_1,C_2$ such that for any $0<\epsilon\le1/2$,
\begin{equation}
\Pr\left(\max_{1\le i\le n}|\mu_i-y_i|>\epsilon\right)\le C_1ne^{-C_2(a+b)\epsilon^4}.
\end{equation}
\end{corollary}

As the main result of this subsubsection, Theorem \ref{thm:j2} is an improvement of Corollary \ref{c:3}.

\begin{theorem} \label{thm:j2}
There exist positive constants $C_1,C_2$ such that for any $\epsilon>0$,
\begin{equation}
\Pr\left(\max_{1\le i\le n}|\mu_i-y_i|>\epsilon\right)\le C_1ne^{-C_2(a+b)\epsilon^2}.
\end{equation}
\end{theorem}

Section 3 of Ref.~\cite{DN09} presents several applications of Theorem \ref{thm:j1}. Most of them can be improved by using Theorem \ref{thm:j2}. We discuss one of them in detail.

Let $\beta$ be a positive constant. Consider the limit $n\to\infty$ with
\begin{equation}
a=\omega(n),\quad a=\Theta(b).
\end{equation}
Let $\delta(\cdot)$ be the Dirac delta. The semicircle law with radius $r$ is a probability distribution on the interval $[-r,r]$ with density function
\begin{equation}
f_\textnormal{SC}(\mu)\propto\sqrt{r^2-\mu^2}.
\end{equation}

\begin{corollary} \label{cor:j}
The empirical distribution
\begin{equation}
f(\mu):=\frac1n\sum_{i=1}^n\delta\left(\mu-\sqrt\frac{a+b}{2abn\beta}\big((a+b)\mu_i+a-b\big)\right)
\end{equation} 
of linearly transformed $\mu_i$ converges weakly to the semicircle law with radius $2$ almost surely.
\end{corollary}

For $\omega(n)=a=o(n^2/\ln n)$, Corollary \ref{cor:j} was proved in Example 3.4 of Ref.~\cite{DN09} using Theorem \ref{thm:j1}. Using Theorem \ref{thm:j2} instead, the same proof becomes valid for any $a=\omega(n)$.

Corollary \ref{cor:j} is very similar to Theorem 2.1 in Ref.~\cite{Nag14}.

\subsubsection{Moments}\label{sec:moments}

Theorem \ref{thm:j2} implies the concentration of any smooth multivariate function of $\mu_1,\mu_2,\ldots,\mu_n$. The main result of this subsubsection is tighter concentration bounds (than those implied by Theorem \ref{thm:j2}) for the first and second moments of the Jacobi ensemble.

Let
\begin{equation}
N:=a+b+\beta(n-1).
\end{equation}
Suppose that $\beta=\Theta(1)$ is a positive constant and that $a+b=\Omega(1)$. In this subsubsection, we are interested in the limit $N\to\infty$. This means that $a+b\to\infty$ or $n\to\infty$ or both.

Let
\begin{equation}
M^\textnormal{J}_1:=\frac1n\sum_{i=1}^n\mu_i,\quad M^\textnormal{J}_2:=\frac1n\sum_{i=1}^n(\mu_i-\e M^\textnormal{J}_1)^2
\end{equation}
be the first and shifted second moments of the Jacobi ensemble. Equation (B.7) of Ref.~\cite{MRW17} implies that
\begin{equation}
\e M^\textnormal{J}_1=\frac{b-a}N,\quad\e M^\textnormal{J}_2=\frac{\beta n(2a+\beta n)(2b+\beta n)}{2N^3}+O(1/N).
\end{equation}
Indeed, $\e M^\textnormal{J}_2$ can be calculated exactly in closed form. The expression is lengthy and simplifies to the above using the Big-O notation.

\begin{theorem} [concentration of moments] \label{thm:moment}
For any $\epsilon>0$,
\begin{equation}
\Pr(|M^\textnormal{J}_1-\e M^\textnormal{J}_1|>\epsilon)=O(e^{-\Omega(Nn\epsilon^2)}),\quad\Pr(|M^\textnormal{J}_2-\e M^\textnormal{J}_2|>\epsilon)=O(e^{-\Omega(N\epsilon)\min\{N\epsilon,n\}}).\label{eq:20}
\end{equation}
\end{theorem}

Let 
\begin{equation}
Y_1:=\frac1n\sum_{i=1}^ny_i,\quad Y_2:=\frac1n\sum_{i=1}^n(y_i-Y_1)^2
\end{equation}
be the mean and variance of the zeros of the Jacobi polynomial. From direct calculation (Appendix \ref{app}) we find that
\begin{equation} \label{eq:app}
Y_1=\frac{b-a}N,\quad Y_2=\frac{\beta(n-1)\big(2a+\beta(n-1)\big)\big(2b+\beta(n-1)\big)}{N^2(2N-\beta)}.
\end{equation}
Hence,
\begin{equation}
\e M^\textnormal{J}_1=Y_1,\quad\e M^\textnormal{J}_2=Y_2+O(1/N).
\end{equation}

\begin{corollary}
For any $\epsilon>0$,
\begin{equation}
\Pr(|M^\textnormal{J}_1-Y_1|>\epsilon)=O(e^{-\Omega(Nn\epsilon^2)}),\quad\Pr(|M^\textnormal{J}_2-Y_2|>\epsilon)=O(e^{-\Omega(N\epsilon)\min\{N\epsilon,n\}}).
\end{equation}
\end{corollary}

\section{Proofs} \label{s:p}

The proofs of Theorems \ref{thm:l3} and \ref{thm:j2} are similar to those of Theorems \ref{thm:l1} and \ref{thm:j1} in Refs.~\cite{DI07, DN09}, respectively, but the error analysis is improved and arguably simpler.

The following lemma will be used multiple times.

\begin{lemma} \label{l:diff}
Let $m$ be an integer and $p_i,q_i$ be numbers such that $|p_i-q_i|\le\delta$ for $i=1,2,\ldots,m$. Then,
\begin{equation}
\left|\prod_{i=1}^mp_i-\prod_{i=1}^mq_i\right|\le\delta\sum_{k=0}^{m-1}\prod_{i=1}^{m-k-1}|p_i|\prod_{j=m+1-k}^m|q_j|.
\end{equation}
\end{lemma}

\begin{proof}
\begin{equation}
\left|\prod_{i=1}^mp_i-\prod_{i=1}^mq_i\right|\le\sum_{k=0}^{m-1}\left|\prod_{i=1}^{m-k}p_i\prod_{j=m+1-k}^mq_j-\prod_{i=1}^{m-k-1}p_i\prod_{j=m-k}^mq_j\right|\le\delta\sum_{k=0}^{m-1}\prod_{i=1}^{m-k-1}|p_i|\prod_{j=m+1-k}^m|q_j|.
\end{equation}
\end{proof}

Let $C$ be a positive constant. For notational simplicity, we will reuse $C$ in that its value may be \emph{different} in different expressions or equations.

\subsection{Laguerre ensemble: Proofs of Theorem \ref{thm:l3} and Fact \ref{fact:M1L}}\label{sec:laguerre-proof}

For Theorem \ref{thm:l3}, it suffices to prove

\begin{theorem} \label{thm:7}
For any $\epsilon>0$,
\begin{equation}
\Pr\left(\frac1{2\alpha}\max_{1\le i\le n}|\lambda_i-x_i|>4\epsilon\right)\le4ne^{-\alpha(\sqrt{1+\epsilon}-1)^2}.
\end{equation}
\end{theorem}

\begin{proof} [Proof of Theorem \ref{thm:7}]
Let $X_{2\alpha},X_{2\alpha-\beta},X_{2\alpha-2\beta},\ldots,X_{2\alpha-(n-1)\beta},Y_\beta,Y_{2\beta},\ldots,Y_{(n-1)\beta}$ be independent non-negative random variables with $X_k^2\sim\chi^2_k$ and $Y_l^2\sim\chi^2_l$. Note that
\begin{equation}
\e(X_k^2)=k,\quad\Var(X_k^2)=2k.
\end{equation}
Lemma A.1 in Ref.~\cite{DI07} gives the tail bound ($\delta$ here and in all probability bounds below is positive)
\begin{equation} \label{eq:ctr}
\Pr(|X_k-\sqrt k|>\delta)\le2(1+\delta/\sqrt k)^ke^{-\delta\sqrt k-\delta^2/2}\le2e^{-\delta^2/2}.
\end{equation}

Let $\mathbf L_{i,j}$ be the element in the $i$th row and $j$th column of a real symmetric $n\times n$ tridiagonal random matrix $\mathbf L$. ``Tridiagonal'' means that $\mathbf L_{i,j}=0$ if $|i-j|>1$. The diagonal and subdiagonal matrix elements are, respectively,
\begin{gather}
\mathbf L_{1,1}=X_{2\alpha}^2,\label{eq:bl1}\\
\mathbf L_{i,i}=X_{2\alpha-(i-1)\beta}^2+Y_{(n+1-i)\beta}^2,\quad i=2,3,\ldots,n,\label{eq:bl2}\\
\mathbf L_{i+1,i}=X_{2\alpha-(i-1)\beta}Y_{(n-i)\beta},\quad i=1,2,\ldots,n-1.\label{eq:al}
\end{gather}
The joint eigenvalue distribution of $\mathbf L$ is given by \cite{DE02} the Laguerre ensemble (Definition \ref{def:le}).

Let $\mathbf L'$ be a real symmetric $n\times n$ tridiagonal deterministic matrix, whose matrix elements are obtained by replacing $X_k^2,Y_l^2$ in Eqs. (\ref{eq:bl1}), (\ref{eq:bl2}) by their expectation values and replacing $X_kY_l$ in Eq.~(\ref{eq:al}) by $\sqrt{\e(X_k^2)\e(Y_l^2)}$, i.e.,
\begin{gather}
    \mathbf L'_{1,1}=2\alpha,\\
    \mathbf L'_{i,i}=2\alpha+(n+2-2i)\beta,\quad i=2,3,\ldots,n,\\
    \mathbf L'_{i+1,i}=\sqrt{\big(2\alpha-(i-1)\beta\big)(n-i)\beta},\quad i=1,2,\ldots,n-1.
\end{gather}
The eigenvalues of $\mathbf L'$ are the zeros of the Laguerre polynomial $L_n^{(2a/\beta-n)}(x/\beta)$ \cite{DI07}.

Let $\|\cdot\|$ denote the operator norm. Let $\mathbf L_{1,0}=\mathbf L'_{1,0}=\mathbf L_{n+1,n}=\mathbf L'_{n+1,n}:=0$. Let $\delta=\sqrt{2\alpha}(\sqrt{1+\epsilon}-1)$. Since
\begin{equation}
    \max_{1\le i\le n}|\lambda_i-x_i|\le\|\mathbf L-\mathbf L'\|\le\max_{1\le i\le n}\{|\mathbf L_{i,i-1}-\mathbf L'_{i,i-1}|+|\mathbf L_{i,i}-\mathbf L'_{i,i}|+|\mathbf L_{i+1,i}-\mathbf L'_{i+1,i}|\},
\end{equation}
it suffices to show that
\begin{equation}
|\mathbf L_{i,i}-\mathbf L'_{i,i}|\le4\alpha\epsilon,\quad|\mathbf L_{i+1,i}-\mathbf L'_{i+1,i}|\le2\alpha\epsilon,\quad\forall i
\end{equation}
under the assumptions that
\begin{equation} \label{eq:40}
|X_k-\sqrt k|\le\delta,\quad|Y_l-\sqrt l|\le\delta,\quad\forall k,l.
\end{equation}
Indeed, (\ref{eq:40}) and Lemma \ref{l:diff} with $m=2$ imply that for any $k,l\le2\alpha$,
\begin{gather}
|X_k^2-k|\le\delta(2\sqrt k+\delta)\le\delta(2\sqrt{2\alpha}+\delta),\\
|X_kY_l-\sqrt{kl}|\le\delta(\sqrt k+\sqrt l+\delta)\le\delta(2\sqrt{2\alpha}+\delta)=2\alpha\epsilon.
\end{gather}
\end{proof}

\begin{proof} [Proof of \Cref{fact:M1L}]
Using the matrix model (\ref{eq:bl1}), (\ref{eq:bl2}), (\ref{eq:al}) from Ref.~\cite{DE02}, we find that
  \begin{equation}
    M_1^\textnormal{L}\sim\frac 1n\sum_{i=1}^n \mathbf L_{i,i} =\frac1n \sum_{i=1}^n X_{2\alpha - (i-1)\beta}^2 + \frac1n\sum_{i=2}^nY_{(n+1-i)\beta}^2
    \sim \frac 1n \chi^2_{2\alpha n}.
  \end{equation}
\end{proof}

\subsection{Jacobi ensemble}

For $k,l>0$, let $Z\sim B(k,l)$ denote a beta-distributed random variable on the interval $[-1,1]$ with probability density function
\begin{equation}
    f_\textnormal{beta}(z)\propto(1-z)^{k-1}(1+z)^{l-1}
\end{equation}
so that
\begin{equation}
    \e Z=\frac{l-k}{k+l}.
\end{equation}
Assume without loss of generality that $k\ge l$. Theorem 8 in Ref.~\cite{ZZ20} gives the tail bound
\begin{equation} \label{eq:btail}
\Pr(Z>\e Z+\delta)\le2e^{-C\min\left\{\frac{k^2\delta^2}{l},k\delta\right\}},\quad\Pr(Z<\e Z-\delta)\le2e^{-\frac{Ck^2\delta^2}{l}}.
\end{equation}
Note that $\Pr(Z>\e Z+\delta)=0$ for $\delta\ge1-\e Z$. In this case, the first inequality above holds trivially. The tail bound (\ref{eq:btail}) implies that
\begin{gather} 
\Pr(|Z-\e Z|>\delta)\le4e^{-Ck\delta^2},\label{eq:btail1}\\
\Pr(Z>\e Z+2\delta\sqrt{1+\e Z}+\delta^2)\le2e^{-Ck\delta^2},\quad\forall\delta>0.\label{eq:btail2}
\end{gather}
Furthermore, for $0<\delta<\sqrt{1+\e Z}$,
\begin{equation} \label{eq:btail3}
\Pr(Z<\e Z-2\delta\sqrt{1+\e Z}+\delta^2)\le 2e^{-Ck\delta^2}.
\end{equation}
(\ref{eq:btail2}) and (\ref{eq:btail3}) imply that
\begin{equation} \label{eq:btail4}
\Pr(|\sqrt{1+Z}-\sqrt{1+\e Z}|>\delta)\le4e^{-Ck\delta^2}.
\end{equation}
Similarly,
\begin{equation} \label{eq:btail5}
\Pr(|\sqrt{1-Z}-\sqrt{1-\e Z}|>\delta)\le4e^{-Ck\delta^2}.
\end{equation}

\subsubsection{Pointwise approximation: Proof of Theorem \ref{thm:j2}}

Let $Z_2,Z_3,Z_4,\ldots,Z_{2n}$ be independent random variables with distribution
\begin{equation}
    Z_i\sim
    \begin{cases}
    B\big(a+(2n-i)\beta/4,b+(2n-i)\beta/4\big),\quad\textnormal{even}~i\\
    B\big(a+b+(2n-1-i)\beta/4,(2n+1-i)\beta/4\big),\quad\textnormal{odd}~i
    \end{cases}
\end{equation}
so that
\begin{equation} \label{eq:zm}
    \e Z_i=\frac1{a+b+(n-i/2)\beta}\times
    \begin{cases}
    b-a,\quad\textnormal{even}~i\\
    \beta/2-a-b,\quad\textnormal{odd}~i
    \end{cases}.
\end{equation}
Let $Z_1:=-1$.

Let $\mathbf J_{i,j}$ be the element in the $i$th row and $j$th column of a real symmetric $n\times n$ tridiagonal random matrix $\mathbf J$. The diagonal and subdiagonal matrix elements are, respectively,
\begin{equation} \label{eq:ab}
\mathbf J_{i,i}=(1-Z_{2i-1})Z_{2i}-(1+Z_{2i-1})Z_{2i-2},\quad\mathbf J_{i+1,i}=\sqrt{(1-Z_{2i-1})(1-Z_{2i}^2)(1+Z_{2i+1})}.
\end{equation}
The joint eigenvalue distribution of $\mathbf J/2$ is given by \cite{KN04} the Jacobi ensemble (Definition \ref{def:je}).

Let $\mathbf J'$ be a real symmetric $n\times n$ tridiagonal deterministic matrix, whose matrix elements are obtained by replacing every random variable $Z_i$ in (\ref{eq:ab}) by $\e Z_i$, i.e.,
\begin{gather}
    \mathbf J'_{i,i}=(1-\e Z_{2i-1})\e Z_{2i}-(1+\e Z_{2i-1})\e Z_{2i-2},\label{eq:50}\\
    \mathbf J'_{i+1,i}=\sqrt{(1-\e Z_{2i-1})\big(1-(\e Z_{2i})^2\big)(1+\e Z_{2i+1})}.\label{eq:51}
\end{gather}
The eigenvalues of $\mathbf J'/2$ are the zeros of the Jacobi polynomial $P_n^{2a/\beta-1,2b/\beta-1}(x)$ \cite{DN09}.

Let $\mathbf J_{1,0}=\mathbf J'_{1,0}=\mathbf J_{n+1,n}=\mathbf J'_{n+1,n}:=0$. Using (\ref{eq:btail1}), (\ref{eq:btail4}), (\ref{eq:btail5}) and since
\begin{equation}
    \max_{1\le i\le n}|\mu_i-y_i|\le\|\mathbf J-\mathbf J'\|/2\le\max_{1\le i\le n}\{|\mathbf J_{i,i-1}-\mathbf J'_{i,i-1}|+|\mathbf J_{i,i}-\mathbf J'_{i,i}|+|\mathbf J_{i+1,i}-\mathbf J'_{i+1,i}|\}/2,
\end{equation}
it suffices to show that
\begin{gather}
    |\mathbf J_{i,i}-\mathbf J'_{i,i}|\le C\epsilon,\quad\forall i,\label{eq:top1}\\
    |\mathbf J_{i+1,i}-\mathbf J'_{i+1,i}|\le C\epsilon,\quad\forall i\label{eq:top2}
\end{gather}
under the assumptions that
\begin{gather}
    |Z_i-\e Z_i|\le\epsilon,\quad\forall i,\label{eq:err1}\\
    |\sqrt{1+Z_i}-\sqrt{1+\e Z_i}|\le\epsilon,\quad|\sqrt{1-Z_i}-\sqrt{1-\e Z_i}|\le\epsilon,\quad\forall i.\label{eq:err2}
\end{gather}
(\ref{eq:top1}) follows from (\ref{eq:err1}) and Lemma \ref{l:diff} with $m=2$. (\ref{eq:top2}) follows from (\ref{eq:err2}) and Lemma \ref{l:diff} with $m=4$.

\subsubsection{Moments: Proof of Theorem \ref{thm:moment}}

Since $N=O(\max\{a+b,n\})$, it suffices to prove that
\begin{gather}
\Pr(|M^\textnormal{J}_1-\e M^\textnormal{J}_1|>\epsilon)=O(e^{-\Omega(a+b)n\epsilon^2}),\label{eq:60}\\
\Pr(|M^\textnormal{J}_1-\e M^\textnormal{J}_1|>\epsilon)=O(e^{-\Omega(n^2\epsilon^2)}),\label{eq:61}\\
\Pr(|M^\textnormal{J}_2-\e M^\textnormal{J}_2|>\epsilon)=O(e^{-\Omega(a+b)\epsilon\min\{N\epsilon,n\}}),\label{eq:62}\\
\Pr(|M^\textnormal{J}_2-\e M^\textnormal{J}_2|>\epsilon)=O(e^{-\Omega(n^2\epsilon^2)}).\label{eq:63}
\end{gather}

We follow the proof of Theorem \ref{thm:j2} and use the same notation. We have proved that
\begin{gather}
\Pr(|\mathbf J_{i,i}-\mathbf J'_{i,i}|>\delta)=O(e^{-\Omega(a+b)\delta^2}),\quad\forall i,\label{eq:54}\\
\Pr(|\mathbf J_{i+1,i}-\mathbf J'_{i+1,i}|>\delta)=O(e^{-\Omega(a+b)\delta^2}),\quad\forall i.\label{eq:65}
\end{gather}
Let $I_n$ be the identity matrix of order $n$. A straightforward calculation using (\ref{eq:ab}) yields
\begin{align}
M^\textnormal{J}_1&=\frac1n\tr\frac{\mathbf J}2=\frac1{2n}\sum_{i=1}^n\mathbf J_{i,i}=\frac1{2n}\left(Z_{2n}-\sum_{i=2}^{2n}Z_{i-1}Z_i\right),\label{eq:66}\\
M^\textnormal{J}_2&=\frac1n\tr((\mathbf J/2-Y_1I_n)^2)=\frac1n\sum_{i=1}^n(\mathbf J_{i,i}/2-Y_1)^2+\frac1{2n}\sum_{i=1}^{n-1}\mathbf J_{i+1,i}^2\label{eq:67}\\
&=Y_1^2-2Y_1M^\textnormal{J}_1+\frac12+\frac{2Z_{2n-1}(1-Z_{2n}^2)+Z_2^2+Z_{2n}^2}{4n}+M',\label{eq:68}
\end{align}
where
\begin{equation} \label{eq:defm}
M':=\frac1{4n}\sum_{i=3}^{2n}\big(2Z_{i-2}(Z_{i-1}^2-1)Z_i+Z_{i-1}^2Z_i^2\big).
\end{equation}

We will use the Chernoff bound multiple times.

\begin{lemma} \label{l:c}
Let $W_1,W_2,\ldots,W_n$ be independent real-valued random variables such that
\begin{equation} \label{eq:tail}
\e W_i=0,\quad\Pr(|W_i|>x)=O\left(e^{-\min\left\{\frac xr,\frac{x^2}{s^2}\right\}}\right),\quad\forall i
\end{equation}
for some $r,s>0$. Then,
\begin{equation}
\Pr\left(\left|\frac1n\sum_{i=1}^nW_i\right|>\delta\right)=O\left(e^{-\Omega(n)\min\left\{\frac\delta r,\frac{\delta^2}{r^2+s^2}\right\}}\right).
\end{equation}
\end{lemma}

Each $W_i$ is a subexponential random variable in that its probability distribution satisfies (\ref{eq:tail}). Thus, Lemma \ref{l:c} is the Chernoff bound for subexponential random variables. For $r=0^+$, $W_i$ becomes a sub-Gaussian random variable, and Lemma \ref{l:c} reduces to the Chernoff bound for sub-Gaussian random variables.

\begin{proof} [Proof of Lemma \ref{l:c}]
The tail bound (\ref{eq:tail}) implies that for any $j>0$,
\begin{multline}
\e(|W_i|^j)=\int_0^\infty\Pr(|W_i|^j>x)\,\mathrm dx=\int_0^\infty jx^{j-1}\Pr(|W_i|>x)\,\mathrm dx\\
=\int_0^\infty jx^{j-1}O(e^{-x/r}+e^{-x^2/s^2})\,\mathrm dx=O\big(r^j\Gamma(j+1)+s^j\Gamma(j/2+1)\big).
\end{multline}
Let $t$ be such that $0<t\le1/(2r)$. Since $\e W_i=0$,
\begin{equation}
\e e^{tW_i}=1+\sum_{j=2}^\infty\frac{t^j\e(W_i^j)}{j!}=1+\sum_{j=2}^\infty O\left((rt)^j+\frac{(st)^j\Gamma(j/2+1)}{j!}\right).
\end{equation}
Using $(st)^j\le(st)^{j-1}+(st)^{j+1}$ for odd $j$,
\begin{multline} \label{eq:mgf}
\e e^{tW_i}=1+\frac{O(rt)^2}{1-rt}+\sum_{j=1}^\infty(st)^{2j}O\left(\frac{\Gamma(j+1/2)}{(2j-1)!}+\frac{j!}{(2j)!}+\frac{\Gamma(j+3/2)}{(2j+1)!}\right)\\
=1+O(rt)^2+O(1)\sum_{j=1}^\infty\frac{(st)^{2j}}{j!}\le e^{c(r^2+s^2)t^2},
\end{multline}
where $c>0$ is a constant. Recall the standard Chernoff argument:
\begin{equation} \label{eq:ca}
\Pr\left(\frac1n\sum_{i=1}^nW_i>\delta\right)=\Pr(e^{t\sum_{i=1}^nW_i}>e^{nt\delta})\le e^{-nt\delta}\e e^{t\sum_{i=1}^nW_i}=\prod_{i=1}^n\e e^{tW_i-t\delta}.
\end{equation}
If $\delta\le c(r^2+s^2)/r$, we choose $t=\frac\delta{2c(r^2+s^2)}$ so that
\begin{equation}
\e e^{tW_i-t\delta}\le e^{-\frac{\delta^2}{4c(r^2+s^2)}}.
\end{equation}
If $\delta>c(r^2+s^2)/r$, we choose $t=1/(2r)$ so that
\begin{equation}
\e e^{tW_i-t\delta}\le e^{\frac{c(r^2+s^2)}{4r^2}-\frac{\delta}{2r}}\le e^{-\frac\delta{4r}}.
\end{equation}
We complete the proof by combining these two cases.
\end{proof}

\begin{lemma} \label{l:c2}
Let $W_1,W_2,\ldots,W_n$ be independent random variables on the interval $[-1,1]$ such that
\begin{equation} \label{eq:tail'}
\e W_i=0,\quad\Pr(|W_i|>x)=O\left(e^{-\min\left\{(r+is)x,\frac{(r+is)^3x^2}{r^2}\right\}}\right),\quad\forall i
\end{equation}
for some $r,s=\Omega(1)$. Then,
\begin{equation}
\Pr\left(\left|\frac1n\sum_{i=1}^nW_i\right|>\delta\right)=O(e^{-\Omega(n^2\delta^2)}).
\end{equation}
\end{lemma}

\begin{proof}
For $i\ge(2t-r)/s$, by replacing (\ref{eq:tail}) with (\ref{eq:tail'}), (\ref{eq:mgf}) implies that
\begin{equation}
\e e^{tW_i}=e^{\frac{O(t^2)}{(r+is)^2}+\frac{O(r^2t^2)}{(r+is)^3}}.
\end{equation}
Since $|W_i|\le1$, we trivially have
\begin{equation}
\e e^{tW_i}\le e^t.
\end{equation}
The Chernoff argument (\ref{eq:ca}) implies that
\begin{multline}
\Pr\left(\frac1n\sum_{i=1}^nW_i>\delta\right)\le e^{-nt\delta}\prod_{i=1}^n\e e^{tW_i}\le e^{-nt\delta}\prod_{1\le i<\frac{2t-r}s}e^t\times\prod_{\max\left\{\frac{2t-r}s,1\right\}\le i\le n}e^{\frac{O(t^2)}{(r+is)^2}+\frac{O(r^2t^2)}{(r+is)^3}}\\
\le e^{-nt\delta+O(t^2)+\sum_{i=1}^\infty\frac{O(t^2)}{(r+is)^2}+\frac{O(r^2t^2)}{(r+is)^3}}=e^{O(t^2)-nt\delta}.
\end{multline}
We complete the proof by choosing $t=c'n\delta$ for a sufficiently small constant $c'>0$.
\end{proof}

\begin{proof} [Proof of Eq.~(\ref{eq:60})]
Using Eq.~(\ref{eq:54}) and Lemma \ref{l:c} with $r=0^+$,
\begin{gather}
\Pr\left(\left|\frac1n\sum_{\textnormal{even}~i}(\mathbf J_{i,i}-\mathbf J'_{i,i})\right|>\epsilon\right)=O(e^{-\Omega(a+b)n\epsilon^2}),\\
\Pr\left(\left|\frac1n\sum_{\textnormal{odd}~i}(\mathbf J_{i,i}-\mathbf J'_{i,i})\right|>\epsilon\right)=O(e^{-\Omega(a+b)n\epsilon^2}).
\end{gather}
Then, Eq.~(\ref{eq:60}) follows from Eq.~(\ref{eq:66}) and the union bound.
\end{proof}

\begin{proof} [Proof of Eq.~(\ref{eq:61})]
The tail bound (\ref{eq:btail1}) implies that
\begin{equation} \label{eq:btail6}
\Pr(|Z_i-\e Z_i|>\delta)=O(e^{-\Omega(a+b+(n-i/2)\beta)\delta^2})
\end{equation}
so that
\begin{equation}
\Pr(|Z_{i-1}Z_i-\e Z_{i-1}\cdot\e Z_i|>\delta)=O\left(e^{-\Omega(a+b+(n-i/2)\beta)\delta\min\left\{\frac{(a+b+(n-i/2)\beta)^2\delta}{(a+b)^2},1\right\}}\right).
\end{equation}
Using Lemma \ref{l:c2},
\begin{gather}
\Pr\left(\left|\frac1n\sum_{\textnormal{even}~i}(Z_{i-1}Z_i-\e Z_{i-1}\cdot\e Z_i)\right|>\epsilon\right)=O(e^{-\Omega(n^2\epsilon^2)}),\\
\Pr\left(\left|\frac1n\sum_{\textnormal{odd}~i}(Z_{i-1}Z_i-\e Z_{i-1}\cdot\e Z_i)\right|>\epsilon\right)=O(e^{-\Omega(n^2\epsilon^2)}).
\end{gather}
Then, Eq.~(\ref{eq:61}) follows from Eq.~(\ref{eq:66}) and the union bound.
\end{proof}

\begin{proof} [Proof of Eq.~(\ref{eq:62})]
Equations (\ref{eq:zm}), (\ref{eq:ab}), (\ref{eq:50}), (\ref{eq:51}) imply that
\begin{gather}
|\mathbf J'_{i,i}/2-Y_1|=O(n/N),\quad|\mathbf J'_{i+1,i}|=O(\sqrt{n/N}),\quad\forall i,\label{eq:59}\\
|\e((\mathbf J_{i,i}/2-Y_1)^2)-(\mathbf J'_{i,i}/2-Y_1)^2|=\frac{O(1)}{a+b},\quad|\e(\mathbf J^2_{i+1,i})-\mathbf J'^2_{i+1,i}|=\frac{O(n)}{(a+b)N},\quad\forall i.\label{eq:71}
\end{gather}
Equations (\ref{eq:54}), (\ref{eq:65}), (\ref{eq:59}) imply that
\begin{gather}
\Pr\big(|(\mathbf J_{i,i}/2-Y_1)^2-(\mathbf J'_{i,i}/2-Y_1)^2|>\delta\big)=O(e^{-\Omega(a+b)\delta\min\{N^2\delta/n^2,1\}}),\quad\forall i,\\
\Pr(|\mathbf J_{i+1,i}^2-\mathbf J'^2_{i+1,i}|>\delta)=O(e^{-\Omega(a+b)\delta\min\{N\delta/n,1\}}),\quad\forall i.
\end{gather}
Using Eq.~(\ref{eq:71}),
\begin{gather}
\Pr\big(|(\mathbf J_{i,i}/2-Y_1)^2-\e((\mathbf J_{i,i}/2-Y_1)^2)|>\delta\big)=O(e^{-\Omega(a+b)\delta\min\{N^2\delta/n^2,1\}}),\quad\forall i,\\
\Pr\big(|\mathbf J_{i+1,i}^2-\e(\mathbf J^2_{i+1,i})|>\delta\big)=O(e^{-\Omega(a+b)\delta\min\{N\delta/n,1\}}),\quad\forall i.
\end{gather}
Using Lemma \ref{l:c},
\begin{gather}
\Pr\left(\left|\frac1n\sum_{\textnormal{even}~i}\big((\mathbf J_{i,i}/2-Y_1)^2-\e((\mathbf J_{i,i}/2-Y_1)^2)\big)\right|>\epsilon\right)=O(e^{-\Omega(a+b)\epsilon\min\{N\epsilon,n\}}),\\
\Pr\left(\left|\frac1n\sum_{\textnormal{odd}~i}\big((\mathbf J_{i,i}/2-Y_1)^2-\e((\mathbf J_{i,i}/2-Y_1)^2)\big)\right|>\epsilon\right)=O(e^{-\Omega(a+b)\epsilon\min\{N\epsilon,n\}}),\\
\Pr\left(\left|\frac1n\sum_{\textnormal{even}~i}\big(\mathbf J_{i+1,i}^2-\e(\mathbf J^2_{i+1,i})\big)\right|>\epsilon\right)=O(e^{-\Omega(a+b)\epsilon\min\{N\epsilon,n\}}),\\
\Pr\left(\left|\frac1n\sum_{\textnormal{odd}~i}\big(\mathbf J_{i+1,i}^2-\e(\mathbf J^2_{i+1,i})\big)\right|>\epsilon\right)=O(e^{-\Omega(a+b)\epsilon\min\{N\epsilon,n\}}).
\end{gather}
Then, Eq.~(\ref{eq:62}) follows from Eq.~(\ref{eq:67}) and the union bound.
\end{proof}

\begin{proof} [Proof of Eq.~(\ref{eq:63})]
The tail bound (\ref{eq:btail6}) implies that
\begin{multline}
\Pr\big(|2Z_{i-2}(Z_{i-1}^2-1)Z_i+Z_{i-1}^2Z_i^2-\e(2Z_{i-2}(Z_{i-1}^2-1)Z_i+Z_{i-1}^2Z_i^2)|>\delta\big)\\
=O\left(e^{-\Omega(a+b+(n-i/2)\beta)\delta\min\left\{\frac{(a+b+(n-i/2)\beta)^2\delta}{(a+b)^2},1\right\}}\right).
\end{multline}
Recall the definition (\ref{eq:defm}) of $M'$. It can be proved in the same way as Eq.~(\ref{eq:61}) that
\begin{equation} \label{eq:100}
\Pr(|M'-\e M'|>\epsilon)=O(e^{-\Omega(n^2\epsilon^2)}).
\end{equation}
Equation (\ref{eq:63}) follows from Eqs. (\ref{eq:61}), (\ref{eq:68}), (\ref{eq:100}) and the union bound.
\end{proof}

\section*{Acknowledgments}

This material is based upon work supported by the U.S. Department of Energy, Office of Science, National Quantum Information Science Research Centers, Quantum Systems Accelerator.  AWH was also supported by NSF grants CCF-1729369 and PHY-1818914 and NTT (Grant AGMT DTD 9/24/20).

\appendix

\section{Proof of Eq.~(\ref{eq:app})} \label{app}

We write the Jacobi polynomial (\ref{eq:jacpoly}) as
\begin{equation}
P_n^{p,q}(y)=\frac{\Gamma(p+q+2n+1)}{2^nn!\Gamma(p+q+n+1)}\left(y^n+\sum_{j=0}^{n-1}c_jy^j\right).
\end{equation}
Let $p=2a/\beta-1$ and $q=2b/\beta-1$. From direct calculation we find that
\begin{align}
c_{n-1}&=\frac{2n(p+n)}{p+q+2n}-n=\frac{n(a-b)}{N},\\
c_{n-2}&=n(n-1)\left(\frac12-\frac{2(p+n)}{p+q+2n}+\frac{2(p+n)(p+n-1)}{(p+q+2n)(p+q+2n-1)}\right)\nonumber\\
&=\frac{n(n-1)\big(2(a-b)^2-\beta N\big)}{2N(2N-\beta)}.
\end{align}
Hence,
\begin{align}
Y_1&=-c_{n-1}/n=(b-a)/N,\\
Y_2&=-Y_1^2+\frac1n\sum_{j=1}^ny_j^2=\frac1n\left(\sum_{j=1}^ny_j\right)^2-Y_1^2-\frac1n\sum_{j\neq k}y_jy_k=(n-1)Y_1^2-\frac{2c_{L-2}}n\nonumber\\
&=\beta(n-1)(1-Y_1^2)/(2N-\beta).
\end{align}

\section{Moments of the Hermite ensemble}
\label{app:more-moments}

Fact \ref{fact:M1L} and Theorem \ref{thm:moment} concern the moments of the Laguerre and Jacobi ensembles, respectively.  For the Hermite ensemble, it is simple to calculate the distributions of the first and second moments exactly. The results are presented here for completeness.

\begin{definition} [Hermite ensemble]
The probability density function of the $\beta$-Hermite ensemble is
\begin{equation} \label{eq:Hermite}
  f_\textnormal{Herm}(\nu_1,\nu_2,\ldots,\nu_n)
  \propto{\prod_{1\le i<j\le n}|\nu_i-\nu_j|^\beta}\prod_{i=1}^ne^{-\nu_i^2/2}.
\end{equation}
\end{definition}

For $\beta=1,2,4$, the Hermite ensemble gives the probability density function of the eigenvalues of an $n\times n$ self-adjoint matrix whose entries are real, complex, or quaternionic Gaussian random variables.

Let
\begin{equation}
  M_1^\textnormal{H}:=\frac1n\sum_{i=1}^n\nu_i,\quad M_2^\textnormal{H}:=\frac1n\sum_{i=1}^n(\nu_i-\e M_1^\textnormal{H})^2 =
  \frac1n\sum_{i=1}^n\nu_i^2
\end{equation}
be the first and second moments of the Hermite ensemble, where we used the fact that $\e M_1^\textnormal{H}=0$.

\begin{fact}
$M_1^\textnormal{H}$ is distributed as $\mathcal N(0,1/n)$, where $\mathcal N(0,\sigma^2)$ denotes the normal distribution with mean $0$ and variance $\sigma^2$. $M_2^\textnormal{H}$ is distributed as $\frac 1n \chi_{n+\beta n(n-1)/2}^2$.
\end{fact}

\begin{proof}
Let $g_1,g_2,\ldots,g_n,X_\beta,X_{2\beta},\ldots,X_{(n-1)\beta}$ be independent random variables with
\begin{equation}
g_i\sim\mathcal N(0,1),\quad X_k^2\sim\chi_k^2,\quad X_k\ge0.
\end{equation}
The eigenvalues of the real symmetric $n\times n$ tridiagonal random matrix
\begin{equation}
\mathbf H = \frac1{\sqrt2}
  \begin{pmatrix}
  \sqrt2g_1 & X_\beta \\
  X_{\beta} &   \sqrt2 g_2 & X_{2\beta} \\
  & X_{2\beta} &   \sqrt2 g_3 & X_{4\beta} \\
  && \ddots &   \ddots & \ddots \\
  &&& X_{(n-2)\beta} & \sqrt2g_{n-1}& X_{(n-1)\beta} \\
   &&&& X_{(n-1)\beta}&   \sqrt2g_{n} 
  \end{pmatrix}
\end{equation}
are distributed according to $f_{\text{Herm}}$ \cite{DE02} so that
\begin{gather}
M_1^\textnormal{H} \sim \frac{1}{n} \tr\mathbf H=\frac1n\sum_{i=1}^ng_i\sim\mathcal N(0,1/n),\\
M_2^\textnormal{H}\sim\frac 1n \tr (\mathbf H^2) = \frac{1}{n}\sum_{i=1}^ng_i^2+\frac1n\sum_{i=1}^{n-1} X_{i\beta}^2\sim\frac 1n \chi_{n+\beta n(n-1)/2}^2.
\end{gather}
\end{proof}

\printbibliography

\end{document}